\documentclass[a4paper]{siamltex}

\usepackage{comment,url,pgf}
\usepackage[english]{babel}
\usepackage[tbtags]{amsmath}
\usepackage{amsfonts,amssymb,mathrsfs,amscd,amsmath,comment}
\def\CC{\mathbb C}
\def\ZZ{\mathbb Z}
\newcommand{\w}{_{_ \mathcal W}}
\newcommand{\f}{_{_\mathcal F}}

\DeclareGraphicsExtensions{.pdf,.png,.jpg,.eps}

\newcommand{\brb}{~\rlap{\rule[0pt]{0.5pt}{0.9ex}}\rule[1.1ex]{0.5pt}{0.9ex}~}

\newtheorem{algorithm}{Algorithm}[section]
\begin{document}

\title{On the exponential of semi-infinite quasi-Toeplitz
  matrices\thanks{The research was carried out with the support of
    GNCS of INdAM}} \author{Dario A. Bini \thanks{Dipartimento di
    Matematica, Universit\`a di Pisa, Largo Bruno Pontecorvo 5, 56127
    Pisa ({\tt dario.bini@unipi.it})} \and Beatrice Meini
  \thanks{Dipartimento di Matematica, Universit\`a di Pisa, Largo
    Bruno Pontecorvo 5, 56127 Pisa ({\tt beatrice.meini@unipi.it})}}

\maketitle

\markright{Exponential of Toeplitz matrices}

\begin{abstract}
  Let $a(z)=\sum_{i\in\mathbb Z}a_iz^i$ be a complex valued function defined for $|z|=1$, such that $\sum_{i\in\mathbb Z}|ia_i|<\infty$,
  and let $E=(e_{i,j})_{i,j\in\ZZ^+}$ be such that
  $\sum_{i,j\in\ZZ^+}|e_{i,j}|<\infty$.  A semi-infinite
  quasi-Toeplitz matrix is a matrix of the kind $A=T(a)+E$, where
  $T(a)=(t_{i,j})_{i,j\in\ZZ^+}$ is the semi-infinite Toeplitz matrix
  associated with the symbol $a(z)$, that is, $t_{i,j}=a_{j-i}$ for
  $i,j\in\mathbb Z^+$.
  We analyze theoretical and computational properties of the
  exponential of $A$. More specifically, it is shown that
  $\exp(A)=T(\exp(a))+F$ where $F=(f_{i,j})_{i,j\in\ZZ^+}$ is such
  that $\sum_{i,j\in\ZZ^+}|f_{i,j}|$ is finite, i.e., $\exp(A)$ is a
  semi-infinite quasi-Toeplitz matrix as well, and an effective algorithm for
  its computation is given. These results can be extended from the
  function $\exp(z)$ to any function $f(z)$ satisfying mild
  conditions, and can be applied to finite quasi-Toeplitz matrices.

\end{abstract}


\section{The problem and its motivation} \label{s1}
Let $a=\{a_k\}_{k\in\mathbb Z}$ be a bi-infinite sequence of complex
numbers, where the index $k$ ranges in the set $\mathbb Z$ of the
relative integers, and define $T(a)=(t_{i,j})_{i,j\in\mathbb Z^+}$ the
semi-infinite Toeplitz matrix such that $t_{i,j}=a_{j-i}$ for $i,j$ in
the set $\mathbb Z^+$ of positive integers.
 
In this paper, we analyze the problem of computing the matrix
exponential of semi-infinite matrices of the form
\[
A=T(a)+E
\]
where $E=(e_{i,j})_{i,j\in\mathbb Z^+}$ is such that
$\sum_{i,j\in\mathbb Z^+}|e_{i,j}|<+\infty$ and $\sum_{k\in\mathbb Z}|ka_k|<+\infty$. 
We refer to this class of
matrices as Quasi-Toeplitz, in short, QT-matrices.

The problem of computing the matrix exponential of Toeplitz and
quasi-Toeplitz matrices is encountered in diverse applications, like
the Erlangian approximation of Markovian fluid queues \cite{dendlat},
\cite{bdlm}, or the discretization of integro-differential equations
with a shift-invariant kernel which describe the pricing of
single-asset options modeled by jump-diffusion processes
\cite{kressner}, \cite{lee-liu-sun}, \cite{merton}, where matrices are
finite but have huge dimensions since their size has an asymptotic
meaning. Another simple example comes from the numerical solution of
the heat equation $\frac{\partial u}{\partial
  t}=\gamma\frac{\partial^2 u}{\partial x^2}$ where the second
derivative in space is discretized with the three-point finite
difference formula, so that the equation is reduced to an ordinary
differential equation of the kind $v'=Av$, being $v$ the vector
function of the values of $u(x,t)$ at the discretization nodes. If the
spatial domain is infinite, then
$A=-\frac\gamma{h^2}\hbox{trid}(-1,2,-1)$ is a semi-infinite
tridiagonal Toeplitz matrix and the solution can be expressed as
$v(t)=\exp(tA)v(0)$.

In other problems, like the analysis of random walks in the quarter
plane \cite{fayolle}, \cite{Takahashi}, or in the tandem Jackson queue
\cite{Sakuma-Miyazawa}, \cite{Motyer-Taylor}, where the number of
states is infinite denumerable, one encounters semi-infinite
probability matrices which are block Toeplitz where the blocks have
the form $T(a)+E$, with $T(a)$ banded Toeplitz and $E$ having a finite
number of nonzero entries.

These applications make it interesting to design effective algorithms
for computing the exponential, and other analytic functions, of
matrices which are quasi-Toeplitz.

Due to its importance, the problem of computing the exponential of a
Toeplitz matrix has been analyzed in several papers among which
\cite{kressner}, \cite{lee-pang-sun}, \cite{pang-sun},
\cite{feng-wu-wei} where different techniques, like using Krylov
subspaces or displacement operators, are used effectively.  The
analysis of the exponential of infinite matrices is considered in
\cite{grimm}, \cite{hochbruck}, while truncating to finite size the
exponential of a general infinite matrix is considered in
\cite{shao2014finite}.  Solving certain matrix polynomial equations
with QT matrices as coefficients has been recently considered in
\cite{bmm}.

\subsection{New results}
In this paper, we provide some theoretical results and some
algorithmic advances concerning the computation of the exponential of
infinite Quasi-Toeplitz matrices $A=T(a)+E$ where
$a(z)=\sum_{i\in\mathbb Z}a_iz^i$ is such that $\sum_{i\in\mathbb Z}|ia_i|$ is finite. 
That is, $a(z)$ and $a'(z)=\sum_{i\in\mathbb Z}ia_iz^{i-1}$ belong to the Wiener class 
\[
\mathcal W=\left\{a(z)=\sum_{i\in\mathbb Z}a_iz^i:\quad \sum_{i\in\mathbb Z}|a_i|<\infty\right\}.
\]
 The approach that we present is
general and can be easily applied to the finite case and extended to
the computation of more general analytic functions. We rely on the
approximation of $\exp(A)$ given by the Taylor series truncation
$S_k=\sum_{i=0}^k\frac1{i!}A^i$, for a sufficiently large $k$. We
exploit the Toeplitz structure of $A$ and of its powers by
representing $S_k$ as $S_k=T(s_k)+F_k$, where
$s_k(z)=\sum_{i=0}^k\frac1{i!}a(z)^i$ is the Taylor series truncation
of $\exp(a)$ and $F_k$ is a suitable correction.

More precisely, we prove that if the function $a(z)$ is such that 
$a(z),a'(z)\in\mathcal W$, then the matrix
exponential $S=\exp(A)$ is well defined by its power series expansion
$\exp(A)=\sum_{i=0}^\infty \frac1{i!}A^i$ and is still of the form
$T(b)+F$ where $b(z)=\sum_{i\in\mathbb Z}b_iz^i=\exp(a(z))$ and
$\sum_{i,j\in\mathbb Z^+}|f_{i,j}|<+\infty$.  That is, $\exp(A)$ has a
{\em Toeplitz component} and a {\em correction component}.  The
former, represented by the coefficients $b_k$, $k\in\mathbb Z$, can be
easily approximated to any precision by applying the evaluation and
interpolation technique at the roots of the unity to the equation
$b(z)=\exp(a(z))$. The correction $F$ can be approximated to any
precision by means of an easily computable recurrence which generates
a sequence of matrices $F_k$ such that $\lim_{k\to\infty} F_k=F$.

This representation of a function of a Toeplitz matrix differs from
those given in \cite{feng-wu-wei}, \cite{kressner},
\cite{lee-pang-sun} \cite{pang-sun} which rely on the displacement
operator and displacement rank.  The idea of representing the infinite
matrix $\exp(A)$ as a Toeplitz part associated with a suitable
function plus a correction having a finite sum of the moduli of its
entries, is fundamental to deal with {\em infinite} matrices by using
only a {\em finite} number of parameters. In fact, the coefficients $g_i$ of a function $g(z)=\sum_{i\in\mathbb Z}g_iz^i$
such that $\sum_{i\in\mathbb Z}|ig_i|<\infty$,
decay to zero as $i\to\pm\infty$ so that $g(z)$ can be approximated by
means of a Laurent polynomial $\widehat
g(z)=\sum_{i=-n_-}^{n_+}g_iz^i$, for $n_-,n_+>0$ sufficiently
large. Moreover, the finiteness of the sum $\sum_{i,j\in\mathbb
  Z^+}|f_{i,j}|$ allows one to approximate $F$ with a finite matrix
generally of small rank.

The decomposition $A=T(a)+E$ is quite natural and widely used in the
analysis of spectral properties of sequences $\{A_n\}$ of $n\times n$
Toeplitz-like matrices where the additive decomposition is given as
the sum of a Toeplitz part, plus a matrix of small rank plus a matrix
of small norm. We refer the reader to \cite{dibenede}, \cite{serra}
for the basic properties and for a list of references in this regard.
The same kind of decomposition is described in \cite[Example 2.28]{bottcher2012introduction}
where, unlike in this paper, it is assumed that the correction $E$ is a compact operator in $\ell^2$. The boundedness of the operator norm $\|E\|_2$ is not enough for our 
computational goals since it does not imply the boundedness of $\sum_{i,j\in\mathbb Z^+}|e_{i,j}|$.

The representation of a matrix as sum of a Toeplitz part and a
correction is described pictorially in Figure \ref{fig1} where the
exponential of $T(a)$ is shown for $a(z)=\sum_{i=-5}^{10}z^i$.  The
two components $T(\exp(a))$, i.e., the Toeplitz part, and the
``compact'' correction $F$ are displayed together with their sum.

The same approach can be applied to the finite case, where the
correction $F$ involves both the north-west and the south-east
corners. In the finite case, the advantage of this representation is
much appreciated if the size of the matrix is larger than the rank of
the correction.

Our approach is more effective when the exponential of the QT-matrix
has a large decay of the band. In fact, in this case, both the degree
of the Laurent polynomial which approximates $\exp(a)$ and the rank of
the correction matrix $F$ are small.  The analysis of the decay of the
band of matrix functions, and more specifically of Toeplitz matrices,
has been recently performed in the papers \cite{BeBo14},
\cite{BeSi15}, \cite{posi16}. The decay of singular values of matrices
with low displacement rank, including positive definite Hankel
matrices, has been performed in \cite{becker}. An immediate
consequence of our results is that the exponential of a quasi-Toeplitz
matrix can be approximated up to
any error $\epsilon$ by a banded matrix.

The acceleration techniques, based on scaling the variable $A$ and
squaring the exponential of the scaled matrix, can be applied to
reduce the number of terms in the power series expansion needed to
reach a sufficiently accurate approximation.  Similarly, our approach
can be combined with Pad\'e approximation to reduce the cost of the
computation.

The algorithm that we have obtained this way has been implemented in
Matlab and tested with some "synthetic" problems and with some
matrices taken from the applications.

From the numerical experiments it turns out that the correction
component $F$ such that $\exp(T(a))=T(\exp(a))+F$, has generally very
low rank. The same holds for the correction $E_k$ such that
$T(a)^k=T(a^k)+E_k$. This property, which will be object of our future
research analysis, is at the basis of the effectiveness of our
algorithm.

Extensions of this approach with the implementation of quasi-Toeplitz
matrix arithmetic are given in \cite{bmm}, applications to general
analytic functions either expressed as power series or as Cauchy
integrals and the algorithmic analysis of the finite case are treated
in \cite{bmm:sbornik}.

\pgfdeclareimage[width=4.5cm]{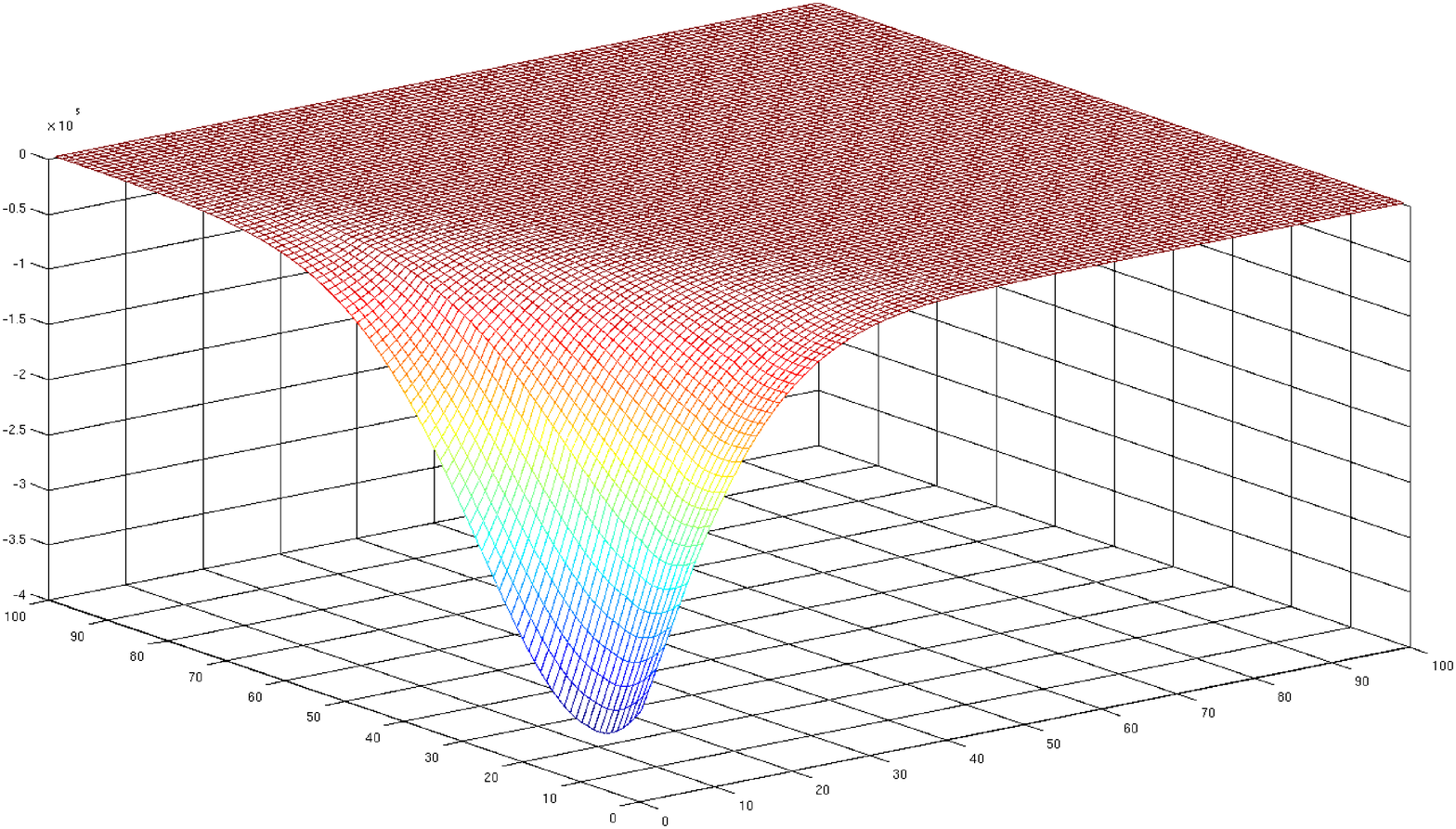}{exp_corr}
\pgfdeclareimage[width=4.5cm]{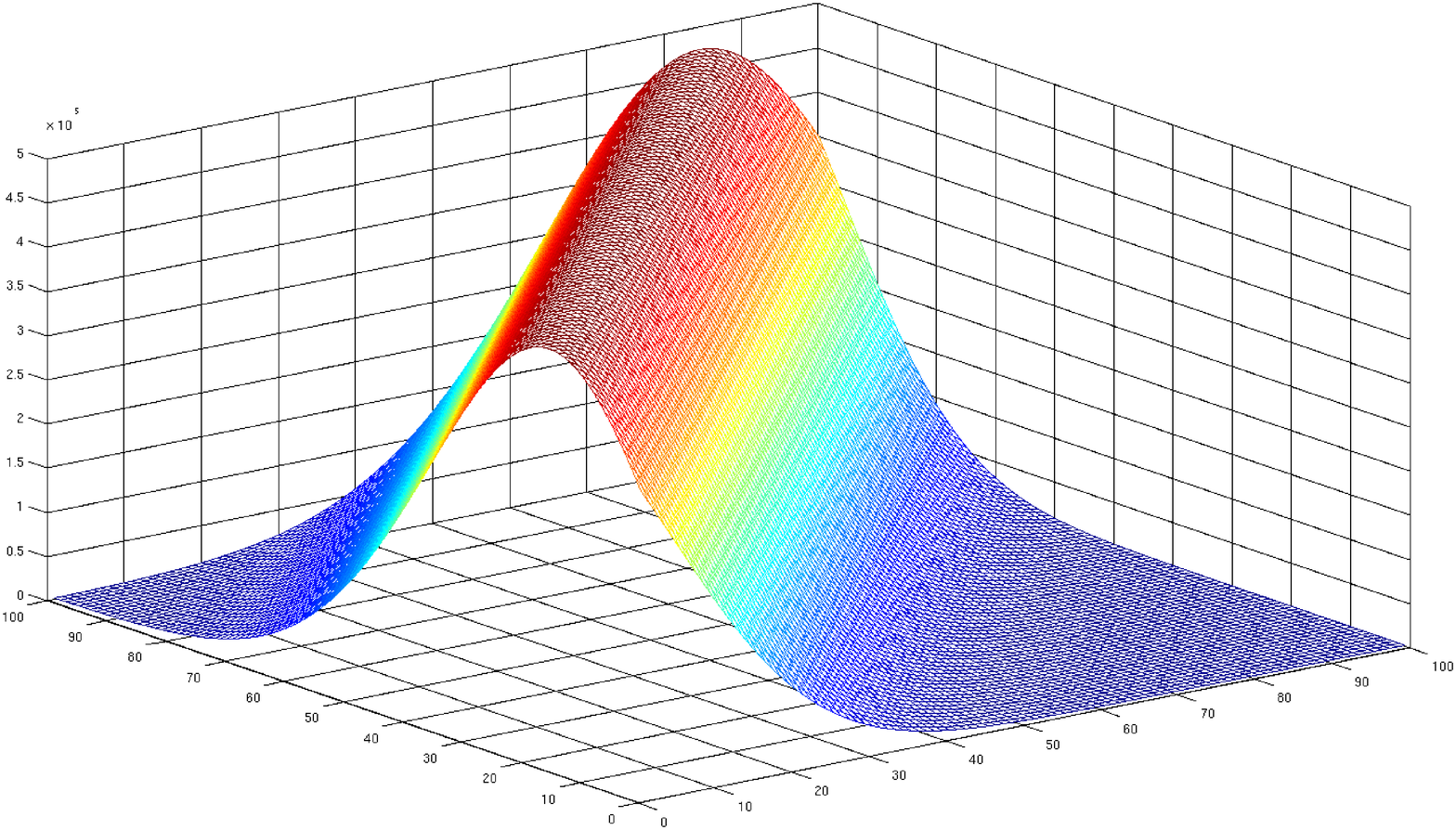}{exp_toep}
\pgfdeclareimage[width=4.5cm]{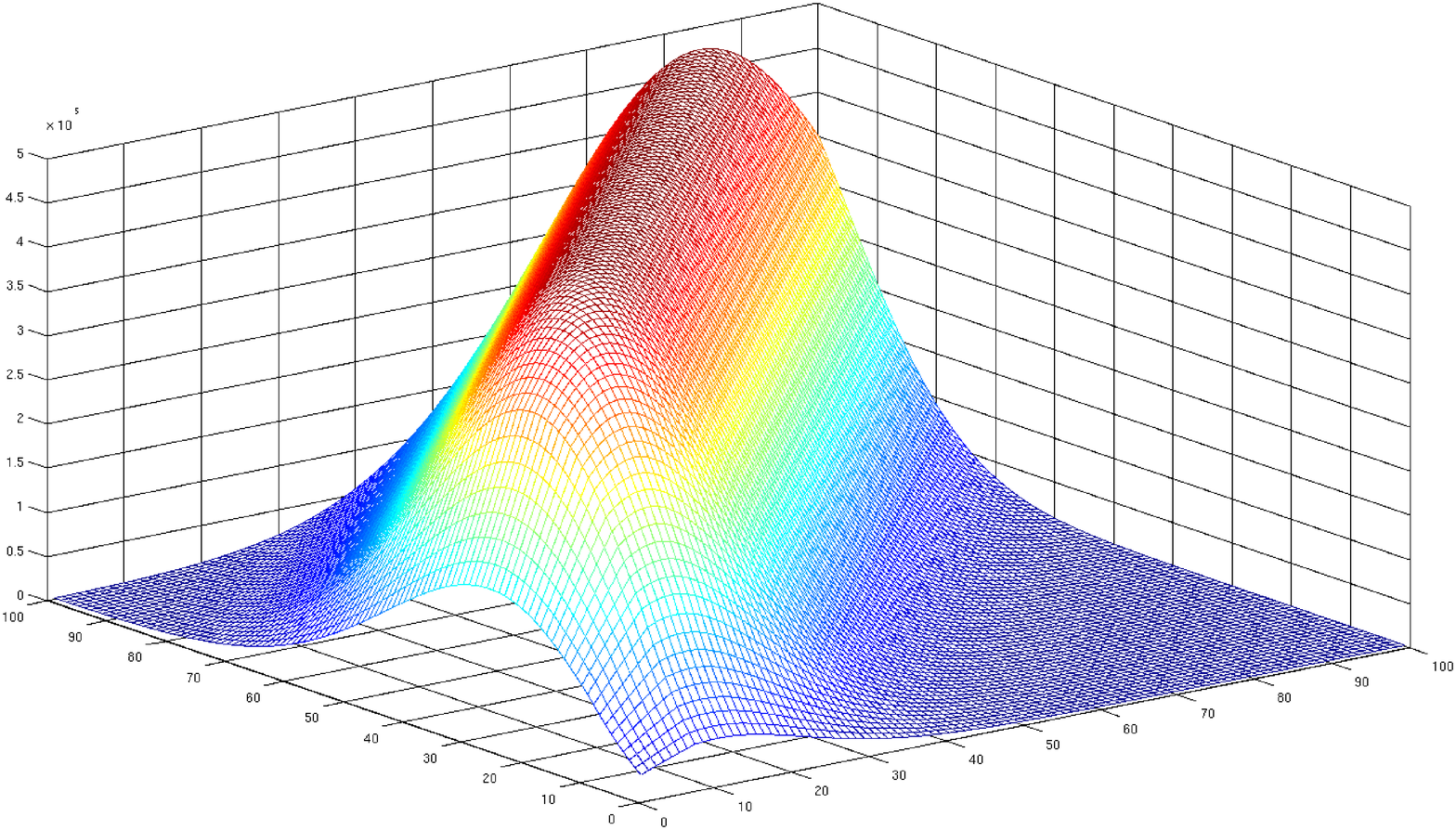}{exp_tot}

\begin{figure}\begin{center}
    \pgfuseimage{exp_tot}\begin{minipage}{0.5cm} \ \\[-2cm]
      =\end{minipage} \pgfuseimage{exp_toep}\begin{minipage}{0.5cm} \
      \\[-2cm] +\end{minipage}\pgfuseimage{exp_corr}
  \end{center}\caption{Decomposition of the matrix exponential as the
    sum of a Toeplitz matrix and of a correction with low numerical
    rank. The matrix is $\exp(T(a))$ where
    $a(z)=\sum_{i=-5}^{10}z^i$}\label{fig1}
\end{figure}

The paper is organized as follows. In Section 2 we recall some
preliminary results concerning semi-infinite Toeplitz matrices,
introduce the norms $\|a\|\w=\sum_{i\in\mathbb Z}|a_i|$ and
$\|F\|\f=\sum_{i,j\in\mathbb Z^+}|f_{i,j}|$, and recall that the
linear space $\mathcal F$ formed by semi-infinite matrices $F$ such
that $\|F\|\f$ is finite is a Banach algebra. In Section 3 we consider
the case of a Toeplitz matrix $T(a)$. We
prove that if $a(z),a'(z)\in\mathcal W$ then $T(a)^i=T(a^i)+E_i$ where $E_i\in\mathcal F$ and give an
explicit relation between $E_i$ and $E_{i-1}$. This relation enables
us to provide a bound to $\|E_i\|\f$ given in terms of $\|a\|\w$ and
of $\|a'\|\w$. In
Section 4 we extend these results to the case of a QT matrix
$A=T(a)+E$. In Section 5 we describe in detail the algorithm for
computing the exponential of a Toeplitz matrix and we outline the case
of a QT matrix. Section 6 reports the results of some numerical
experiments while Section 7 draws some final remarks and conclusions.

\section{Preliminaries}
We recall the basic definition of matrix exponential of an $n\times n$
matrix and the main properties of semi-infinite Toeplitz matrices
which will be used in our analysis.

For an $n\times n$ matrix $A$, it is well known that the series
$\exp(A):=\sum_{i=0}^{+\infty}\frac 1{i!}A^i$ is convergent and
defines the matrix exponential of $A$.  We refer to the book by
N.~Higham \cite{high:book} for the concept of matrix function and for
more details on the matrix exponential.  Indeed, defining $S_k$ the
partial sum
\begin{equation}\label{eq:Ak}
S_k=\sum_{i=0}^k\frac 1{i!}A^i,
\end{equation}
and the remainder $R_k$ of the series as $R_k=\sum_{i=k+1}^\infty\frac
1{i!}A^i$, for any matrix norm $\|\cdot\|$ such that $\|A^2\|\le\|A\|^2$ 
it follows that
\begin{equation}\label{eq:R}
\|R_k\|=\|\sum_{i=k+1}^{\infty}\frac
1{i!}A^i\|\le\sum_{i=k+1}^{\infty}\frac 1{i!}\|A\|^i
\end{equation}
 so that $\lim_{k\to\infty}\|R_k\|=0$ which implies the convergence of
 the sequence $S_k$.

This property is still valid if $A=(a_{i,j})_{i,j\in\ZZ^+}$ is a
semi-infinite matrix provided that $A$ belongs to a Banach algebra
$\mathcal A$, that is an algebra endowed with a sub-multiplicative
norm $\|\cdot\|$, such that $\|AB\|\le\|A\|\cdot\|B\|$ for any
$A,B\in\mathcal A$, which makes it a Banach space.
Indeed, for $A\in\mathcal A$, consider the sequence $\{S_k\}_k$ defined in
\eqref{eq:Ak}. For $i>j$ we have
\[
\|S_i-S_j\|\le\sum_{h=j+1}^{i}\frac1{h!}\|A^h\|\le
\sum_{h=j+1}^{i}\frac1{h!}\|A\|^h.
\]
From this bound it follows that for any $\epsilon>0$ there exists
$k>0$ such that $\|S_i-S_j\|\le\epsilon$ for any $i>j\ge k$. That is,
$\{S_k\}_k$ is a Cauchy sequence. Since by definition of Banach space,
the Cauchy sequences in $\mathcal A$ have a limit in $\mathcal A$,
there exists a matrix $L$ such that $\lim_{k\to\infty}\|S_k-L\|=0$. We
denote $L=\exp(A)$. Thus, the matrix exponential is well defined also
for a semi-infinite matrix $A$ provided that it belongs to a Banach
algebra with respect to some norm.

We now recall some results concerning infinite Toeplitz matrices. For
more details on this topic we refer the reader to the book by
B\"ottcher and Grudsky \cite{bg:book}.

Let $\mathcal W=\{a(z)=\sum_{i\in\ZZ}a_iz^i:\quad
\sum_{i\in\ZZ}|a_i|<+\infty\}$ denote the Wiener algebra formed by
Laurent power series, defined on the unit circle $\mathbb
T=\{z\in\CC:\quad |z|=1\}$, whose coefficients have a finite sum of
their moduli. It is well known that $\mathcal W$ endowed with the norm
$\|a\|_{\mathcal W}=\sum_{i\in\ZZ}|a_i|$ is a Banach algebra. For
$a(z)\in\mathcal W$ denote $T(a)$ the semi-infinite Toeplitz matrix
whose entries $t_{i,j}$ are such that $t_{i,j}=a_{j-i}$ for
$i,j\in\ZZ^+$, where $\ZZ^+$ denotes the set of positive
integers. Denote also $a_+(z)$ and $a_-(z)$ the power series defined
by $a_+(z)=\sum_{i\in\ZZ^+}a_iz^i$ and
$a_-(z)=\sum_{i\in\ZZ^+}a_{-i}z^{i}$ so that
$a(z)=a_0+a_+(z)+a_-(z^{-1})$. Finally, given the power series
$b(z)=\sum_{z\in\ZZ^+}b_iz^i$ define $H(b)=(h_{i,j})$ the Hankel
matrix such that $h_{i,j}=b_{i+j-1}$, for $i,j\in\ZZ^+$.

Any semi-infinite matrix $S=(s_{i,j})_{i,j\in\mathbb Z^+}$ can be
viewed as a linear operator, acting on semi-infinite vectors
$v=(v_i)_{i\in\mathbb Z^+}$, which maps the vector $v$ onto the vector
$u$ such that $u_i=\sum_{j\in\mathbb Z^+}s_{i,j}v_j$, provided that
the summations are finite.  For any $p\ge 1$, included $p=\infty$, we
may define the Banach space $\ell^p$ formed by all the semi-infinite
vectors $v=(v_i)_{i\in\mathbb Z^+}$ such that
$\|v\|_p=(\sum_{i\in\ZZ^+}|v_i|^p)^{\frac 1p}<\infty$, where for
$p=\infty$ we have $\|v\|_\infty=\sup_{i\in\ZZ^+}|v_i|$. It is well
known that these norms induce the corresponding operator norms
$\|S\|_p=\sup_{\|v\|_p=1}\|Sv\|_p$ which are sub-multiplicative, i.e.,
$\|AB\|_p\le\|A\|_p\|B\|_p$ for any semi-infinite matrices $A$, $B$
with finite $\ell^p$ norm, and that the linear space formed by the
latter semi-infinite matrices forms a Banach algebra.

We may wonder if the matrices $T(a)$, $H(a_+)$ and $H(a_-)$ define
linear operators acting on the Banach space $\ell^p$ and if they have
a finite operator norm.  The answer to this question is given by the
following result of \cite{bg:book} which relates the matrix $T(a)T(b)$
with $T(ab)$, $H(a_-)$ and $H(a_+)$.

\begin{theorem}\label{thm1}For $a(z),b(z)\in\mathcal W$ let 
$c(z)=a(z)b(z)$. Then we have
\[
T(a)T(b)=T(c)-H(a_-)H(b_+).
\]
Moreover,  for any $p\ge 1$, including $p=\infty$, we have 
\[
\|T(a)\|_p\le \|a\|\w,\quad \|H(a_-)\|_p\le\|a_-\|\w,\quad 
\|H(b_+)\|_p\le\|b_+\|\w.
\] 
\end{theorem}

The above result implies that the product of two Toeplitz matrices can
be written as a Toeplitz matrix plus a correction whose $\ell^p$-norm
is bounded by $\|a_-\|\w\|b_+\|\w\le\|a\|\w \|b\|\w$. 

For a semi-infinite matrix $S$ we introduce the norm
$\|S\|\f=\sum_{i,j\in\ZZ^+}|s_{i,j}|$, which coincides with the
$\ell^1$ norm if we consider the matrix $S$ as the vector $s_{i,j}$
where the pairs $(i,j)$ are ordered in a triangular fashion starting
from the leftmost top corner $(1,1)$.  We denote by $\mathcal{F}$ the
set of semi-infinite matrices such that $\|S\|\f<\infty$. The next
result states that $\|\cdot\|\f$ is sub-multiplicative in
$\mathcal{F}$ and that $\mathcal F$ is a Banach algebra.

\begin{theorem}\label{thm:3.1}
  For $A,B\in\mathcal F$ it holds $\|AB\|\f\le\|A\|\f\cdot\|B\|\f$,
  moreover, $\mathcal F$ is a Banach algebra.
\end{theorem}
\begin{proof}
  Let $C=AB$ so that $c_{i,j}=\sum_{r\in\ZZ^+}a_{i,r}b_{r,j}$. Since
  $\sum_{i,r\in\ZZ^+}|a_{i,r}|<+\infty$, then
  $\alpha_i=\sum_{r\in\ZZ^+}|a_{i,r}|<+\infty$ and
  $\sum_{i\in\ZZ^+}\alpha_i=\|A\|\f$. Similarly,
  $\beta_j=\sum_{r\in\ZZ^+}|b_{r,j}|<+\infty$ and
  $\sum_{j\in\ZZ^+}\beta_j=\|B\|\f$. Whence we obtain
\[
|c_{i,j}|\le \sum_{r\in\ZZ^+}|a_{i,r}b_{r,j}|\le \alpha_i\beta_j
\]
since, in general $\sum_{r\in\ZZ^+} x_r
y_r\le(\sum_{r\in\ZZ^+}x_r)(\sum_ {r\in\ZZ^+}y_r)$ for any $x_r,y_r\ge
0$ such that $\sum_{r\in\ZZ^+}x_r$ and $\sum_{r\in\ZZ^+}y_r$ are
bounded.  Thus we get
\[
\sum_{i,j\in\ZZ^+}|c_{i,j}|\le
\sum_{i\in\ZZ^+}\alpha_i\sum_{j\in\ZZ^+}\beta_j=\|A\|\f\cdot\|B\|\f.
\]
In order to prove completeness of $\mathcal F$, observe that the norm
$\|\cdot\|_{\mathcal F}$ corresponds to the $\ell^1$ norm in the space
of infinite sequences having finite sum of their moduli. This way, the
space $\mathcal F$ actually coincides with $\ell^1$, which is a Banach
space.
\end{proof}

The following result will be useful to prove boundedness properties of
the exponential of Toeplitz matrices.

\begin{lemma}\label{lem1} If $E=(e_{i,j})\in\mathcal{F}$ and
  $a(z)\in\mathcal{W}$, then
 $\|T(a)E\|\f\le \|a\|\w\|E\|\f$.
\end{lemma}
\begin{proof}
  Let $V=T(a)E$ so that $v_{i,j}=\sum_{r\in\ZZ^+}a_{r-i}e_{r,j}$.
  Observe that for any $j,r\in\mathbb Z^+$, one has $\sum_{i\in\mathbb
    Z^+}|a_{r-i}e_{r,j}|\le \sum_{k\in\mathbb Z}|a_k|\cdot
  |e_{r,j}|=\|a\|\w |e_{r,j}|$. From this inequality we find that

\[\begin{split}
\|V\|\f &=\sum_{i,j\in\mathbb Z^+}|v_{i,j}|
\le\sum_{i,j\in\ZZ^+}\sum_{r\in\ZZ^+}|a_{r-i}e_{r,j}|
= \sum_{r,j\in\ZZ^+}\sum_{i\in\ZZ^+}|a_{r-i}e_{r,j}|\\ 
&\le
\|a\|\w\sum_{r,j\in\mathbb Z^+}|e_{r,j}|=\|a\|\w\|E\|\f.
\end{split}
\]
\end{proof}

\section{Exponential of a semi-infinite Toeplitz matrix}\label{sec:toep}
In this section we study the properties of the exponential of a
semi-infinite Toeplitz matrix, by relating in particular $\exp(T(a))$
to $T(\exp(a))$.

Let $a\in\mathcal W$ and consider the associated semi-infinite Toeplitz
matrix $T(a)$. Since, for Theorem \ref{thm1}, $T(a)$
belong to the Banach algebra of linear operators over $\ell^p$, then
$\exp(T(a))$ is well defined. Moreover, we may write
\[
\exp(T(a))=\sum_{i=0}^\infty\frac1{i!}T(a)^i.
\]
From Theorem \ref{thm1} and for the monotonicity of the function
$\exp(z)$ we have
\[
\|\exp(T(a))\|_p\le\sum_{i=0}^\infty\frac1{i!}\|T(a)\|_p^i=\exp(\|T(a)\|_p)
\le\exp(\|a\|\w).
\]

Now we will take a closer look at $\exp(T(a))$ and relate it to
$T(\exp(a))$. Since $\mathcal W$ is a Banach algebra, the function
$\exp(z)$ is well defined over $\mathcal W$ and we have
\[
\exp(a(z))=\sum_{i=0}^{+\infty}\frac1{i!}a(z)^i.
\]

We first relate $T(a)^i$ to $T(a^i)$, for $i\ge 2$.  From Theorem
\ref{thm1} we may write $T(a)^2=T(a^2)+E_2$, where
$E_2=-H(a_-)H(a_+)$.  For a general $i\ge 0$ define $E_i$ as
\begin{equation}\label{eq:Ei}
E_i=T(a)^i-T(a^i),
\end{equation}
where $E_0=0$, $E_1=0$. Then we have the following

\begin{theorem}\label{th:pow}
Let $a\in\mathcal W$ and let $E_i=T(a)^i-T(a^i)$, for $i\ge 1$.
Then
\begin{equation}\label{eq:Ek}
\begin{split}
&E_{i}=T(a)E_{i-1}-H(a_-)H((a^{i-1})_+),\quad i\ge 2,\\
&E_1=0.
\end{split}
\end{equation}
Moreover, for any $i\ge 1$ and any integer $p\ge 1$, included $p=\infty$,
\begin{equation}\label{eq:normE1}
\|E_i\|_p\le (i-1)\|a\|\w^i.
\end{equation}
\end{theorem}

\begin{proof}From the equation $T(a)^{i}=T(a)T(a)^{i-1}$ and from
  Theorem \ref{thm1} we obtain
\[\begin{split}
T(a)^i=&T(a)T(a)^{i-1}=T(a)[T(a^{i-1})+E_{i-1}]\\
=&
T(a^i)-H(a_-)H((a^{i-1})_+)+T(a)E_{i-1}\\
=&T(a^i)+E_i,
\end{split}
\]
with $E_i=-H(a_-)H((a^{i-1})_+)+T(a)E_{i-1}$.
Whence we deduce recurrence \eqref{eq:Ek}.
Moreover, for any $\ell^p$-norm, since $\|a_+\|\w\le\|a\|\w$,
$\|a_-\|\w\le\|a\|\w$ and $\|a^i\|\w\le\|a\|\w^i$, applying once again
Theorem \ref{thm1}, from \eqref{eq:Ek} we obtain
\[
\|E_i\|_p\le \|a\|\w \|E_{i-1}\|_p+\|a\|\w\|a^{i-1}\|\w\le 
\|a\|\w \|E_{i-1}\|_p+\|a\|\w^i.
\]
By using an induction argument we arrive at the bound
\eqref{eq:normE1}.
\end{proof}

Now define $S_k$, $F_k$ and $G_k$ as follows
\[\begin{split}
&S_k=\sum_{i=0}^k\frac 1{i!}T(a)^i=G_k+F_k\\
&G_k=\sum_{i=0}^k\frac1{i!}T(a^i),\quad F_k=\sum_{i=0}^k\frac1{i!}E_i.
\end{split}
\]

Observe that
$G_k=\sum_{i=0}^k\frac1{i!}T(a^i)=T(\sum_{i=0}^k\frac1{i!}a^i)$ is
such that $\lim_{k\to\infty}G_k=T(\exp(a))$. Thus, since
$\lim_{k\to\infty} S_k=\exp(T(a))$, then there exists the limit
\begin{equation}\label{eq:expT}
\lim_{k\to\infty}F_k=\exp(T(a))-T(\exp(a))=:F.
\end{equation}

This way, we may write the exponential of a semi-infinite Toeplitz matrix
$T(a)$ associated with a symbol $a\in\mathcal W$ in the form
\[
\exp(T(a))=T(\exp(a))+F
\]
where $F=\sum_{i=0}^\infty\frac1{i!}E_i$ has $\ell^p$ norm bounded by
\begin{equation}\label{eq:bound}
\begin{split}
\|F\|_p\le& \|\exp(T(a))\|_p+\|T(\exp(a))\|_p\le\exp(\|a\|\w)+\|\exp(a)\|\w\\
\le&
2\exp(\|a\|\w)
\end{split}\end{equation}
since $\|\exp(a)\|\w\le\exp(\|a\|\w)$.

Now, our next step is to prove a stronger property. That is, we show that under
certain conditions, the correction $F$ is such that
$\|F\|\f<\infty$. This property is stronger than $\|F\|_p<\infty$ and
is very useful from the computational point of view since it implies
that for any $\epsilon>0$ there exists $k$ such that
$|f_{i,j}|\le\epsilon$ for any $i,j>k$. This bound allows us to
represent $F$, up to within any given error bound, by using a finite
number of parameters.

Let $a(z)\in\mathcal W$ be such that $a'(z)\in\mathcal W$ where $a'(z)=\sum_{i\in\mathbb Z}ia_iz^{i-1}$, so that $\sum_{i\in\mathbb Z}|a_i|$ and $\sum_{i\in\mathbb Z}|ia_i|$ are finite. Moreover, observe that if $a(z),a'(z)\in\mathcal W$ and $b(z),b'(z)\in\mathcal W$, then also $c(z),c'(z)\in\mathcal W$ for $c(z)=a(z)+b(z)$ and for $c(z)=a(z)b(z)$. This property enables us to
prove the following

\begin{theorem}\label{th2}
Let $a(z),a'(z)\in\mathcal W$. 
Then  for any $k\ge 1$ we have
\[
\|H(a_-)H((a^{k-1})_+)\|\f\le
(k-1)\|a\|\w^{k-2}\|a'\|\w^2.
\]
\end{theorem}
\begin{proof}
 We have
 $\|H(a_-)\|\f=\sum_{i,j,\in\ZZ^+}|a_{1-i-j}|=\sum_{h\in\ZZ^+}h|a_{-h}|\le\|a'\|\w$
 which is finite since $a'(z)\in\mathcal W$. Similarly,
 $\|H((a^{k-1})_+)\|\f\le \|(a^{k-1})'\|\w<\infty$ since both the functions $a^{k-1}(z)$ and 
 $(a^{k-1}(z))'$ belong to $\mathcal W$. Thus, for the matrix product
 $L_k=H(a_-)H((a^{k-1})_+)$ we find that
\[
\|L_k\|\f\le
\|H(a_-)\|\f\cdot\|H((a^{k-1})_+\|\f\le
\|a'\|\w\|(a^{k-1})'\|\w.
\] Now, since
$(a^{k-1}(z))'=(k-1)a^{k-2}(z)a'(z)$, we have 
\[
\|(a^{k-1})'\|\w\le
(k-1)\|a\|\w^{k-2}\|a'\|\w.
\]
 Thus we get
$\|L_k\|\f\le(k-1)\|a\|\w^{k-2}\|a'\|\w^2$.
\end{proof}

From the above result and from Lemma \ref{lem1}, we deduce the main
theorem of this section

\begin{theorem}\label{mainthm} Under the assumptions of 
Theorem \ref{th2}, for the matrices $E_i$ of \eqref{eq:Ei} and
 $F=\lim_k F_k$, with $F_k=\sum_{i=0}^k \frac1{i!}E_i$, we have
\[\begin{split}
&\|E_i\|\f\le \frac{i(i-1)}2 \|a'\|\w^2\|a\|^{i-2}\w,~~i\ge 0,\\
&\|F\|\f\le\frac12 \|a'\|\w^2\exp(\|a\|\w).
\end{split}\]
\end{theorem}
\begin{proof}
From Theorem \ref{th2} and from \eqref{eq:Ek} we have
\[
\|E_i\|\f\le\|T(a)E_{i-1}\|\f+(i-1)\|a\|\w^{i-2}\|a'\|\w^2,\quad i\ge 2,
\]
where $E_0=E_1=0$.
From Lemma \ref{lem1} we deduce that
\[
\|E_i\|\f\le  \|a\|\w\|E_{i-1}\|\f +(i-1)\|a\|\w^{i-2}\|a'\|\w^2.
\]
Therefore, by using the induction argument we arrive at
\[
\|E_i\|\f\le \frac{i(i-1)}2 \|a'\|\w^2\|a\|\w^{i-2},\quad i\ge 2,
\]
which proves the first bound.
This implies that
\[
\|F\|\f\le \sum_{i=0}^\infty\frac1{i!}\|E_i\|\f\le\frac12 \|a'\|\w^2\exp(\|a\|\w),
\]
which completes the proof.
\end{proof}

This way, we find that $\exp(T(a))=T(\exp(a))+F$, where $F$ is such
that $\|F\|\f<+\infty$. In different words, the exponential of a
semi-infinite Toeplitz matrix is a semi-infinite Toeplitz matrix, up
to a correction belonging to the set $\mathcal F$. This property,
combined with equation \eqref{eq:Ek}, enables us to design algorithms
for computing the exponential of a Toeplitz matrix expressed in the
form $T(\exp(a))+F$. This is possible since both $\exp(a)$ and $F$ can
be represented up to an arbitrarily small error, by means of a finite
number of parameters. We will see the design and analysis of this
class of algorithms in the next Section \ref{sec:algo}.

Observe also that while the relation $F=\exp(T(a))-T(\exp(a))$ is
useful to provide the upper bound
$\|F\|_p\le\|\exp(T(a))\|_p+\|T(\exp(a))\|_p\le 2\exp(\|a\|\w)$ as we
did in \eqref{eq:bound}, it cannot be used to provide finite upper
bounds to $\|F\|\w$ since any nonzero Toeplitz matrix $T(a)$ has
infinite $\mathcal W$ norm.

Before dealing with the algorithmic aspects of this problem, we wish
to extend our analysis to the case of a matrix $A=T(a)+E$ where
$\|E\|\f<+\infty$.  Also in this case we will prove that $\exp(A)$ can
be written in the form $T(\exp(a))+D$ where $\|D\|\f<+\infty$.
Matrices of the kind $T(a)+E$ are encountered in applications related
to the analysis of certain stochastic processes.

\section{A generalization}\label{sec:gener}
Let $A=T(a)+E$, where $a(z),a'(z)\in\mathcal W$,  and $E$ is any semi-infinite matrix belonging to the set
$\mathcal{F}$. The argument used in Section \ref{sec:toep} can be
applied to provide an expression to $\exp(A)$. In fact, we may write $
\exp(A)=S_k+R_k $ where $S_k=\sum_{i=0}^k\frac 1{i!}A^i$ and
$R_k=\sum_{i=k+1}^\infty \frac 1{i!}A^i$, so that $S_k=G_k+F_k$ with
 \[
 G_k=\sum_{i=0}^k\frac 1{i!}T(a^i),~~~F_k=\sum_{i=0}^k\frac 1{i!}D_i,
 \]
 where
 \begin{equation}\label{eq:Dk0}
D_i= A^i-T(a^i),~~i\ge 0.
\end{equation}
Observe that $D_0=0$, $D_1=E$. 
As in the previous section, there exists 
\[
F=\lim_{k\to\infty}F_k=\lim_{k\to\infty}S_k-\lim_{k\to\infty}G_k=\exp(A)-T(\exp(a)).
\]
Our goal is to estimate $\| F\|_\mathcal{F}$ and to show that $\|
F\|_\mathcal{F}<\infty$.

From \eqref{eq:Dk0} we deduce that, for $i\ge 1$,
\[\begin{split}
A^i=&(T(a)+E)(T(a^{i-1})+D_{i-1})\\
=&
T(a)T(a^{i-1})+ET(a^{i-1})+AD_{i-1}
\end{split}\]
and,
 in view of Theorem \ref{thm1}, it follows that
\[
A^i=T(a^{i})-H(a_-)H((a^{i-1})_+)+ET(a^{i-1})+AD_{i-1}.
\]
Hence,  we obtain
\begin{equation}\label{eq:Dk}
D_i= 
A D_{i-1}-H(a_-)H((a^{i-1})_+)+ET(a^{i-1}).
\end{equation}
In view of Lemma \ref{lem1}, we have
$\|T(a)D_{i-1}\|\f\le\|a\|\w\|D_{i-1}\|\f$ so that
\[
\|AD_{i-1}\|\f=\|T(a)D_{i-1}+ED_{i-1}\|\f\le (\|a\|\w+\|E\|\f)\|D_{i-1}\|\f.
\]
Therefore, from the above inequality and from Theorem \ref{th2},
  in view of \eqref{eq:Dk}, we obtain
\[\begin{split}
  \|D_i\|\f\le
  &(\|a\|\w+\|E\|\f)\|D_{i-1}\|\f+(i-1)\|a\|\w^{i-2}\|a'\|\w^2+
  \|a\|\w^{i-1}\|E\|\f,
  \\
  \le &\xi\|D_{i-1}\|\f +\gamma_i,~~i\ge 1,
\end{split}\]
where
\begin{equation}\label{eq:gammai}
  \xi=\|a\|\w+\|E\|\f,\quad \gamma_i=
  (i-1)\|a\|\w^{i-2}\|a'\|\w^2+\|a\|\w^{i-1}\|E\|\f,~~i\ge1,
\end{equation}
and $\|D_0\|\f=0$, $\|D_1\|\f=\|E\|\f$. Thus we may bound $\|D_i\|\f$
with the value that the polynomial $p(z)=\sum_{j=0}^{i-1}
z^j\gamma_{i-j}$ takes at $\xi=\|a\|\w+\|E\|\f$, i.e,
\begin{equation}\label{eq:Di}
\|D_i\|\f\le \sum_{j=0}^{i-1}
\xi^j\gamma_{i-j},\quad \xi=\|a\|\w+\|E\|\f.
\end{equation}

Thus, concerning the sequence $F_k$, from \eqref{eq:Di} we obtain 
\[
\|F_k\|\f\le\sum_{i=1}^k\frac1{i!}\|D_i\|\le
\sum_{i=1}^k\sum_{j=0}^{i-1}\frac1{i!}\xi^j\gamma_{i-j}.
\]
For notational simplicity set $\alpha=\|a\|\w$, $\beta=\|E\|\f$ so
that $\xi=\alpha+\beta$ and $\gamma_k=
(k-1)\alpha^{k-2}\|a'\|\w^2+\alpha^{k-1}\beta$. Then, since
$\alpha\le\xi$, we have
$\gamma_{k}\le(k-1)\xi^{k-2}\|a'\|\w^2+\xi^{k-1}\beta$. Whence we
deduce that

\[\begin{split}
\|F_k\|\f\le&\sum_{i=1}^k
\sum_{j=0}^{i-1}\frac1{i!}\left[\|a'\|\w^2(i-j-1)\xi^{i-2}+
\beta \xi^{i-1}\right]\\
=&
\|a'\|\w^2\sum_{i=1}^k\frac12 \frac{i(i-1)}{i!}\xi^{i-2}+
\beta\sum_{i=1}^k\frac{i}{i!}\xi^{i-1}\\
\le &
\frac12 \|a'\|\w^2\exp(\xi)+\beta \exp(\xi).
\end{split}\]

 Thus we may conclude with the following

\begin{theorem}\label{mainthm1} Under the assumptions of Theorem \ref{th2}, 
  for the matrices $D_i$ defined in \eqref{eq:Dk0}, and
  $F=\lim_{k\to\infty} F_k$, where $F_k=\sum_{i=0}^k \frac1{i!}D_i$,
  we have
\[
\begin{split}
  & \|D_i\|\f\le \sum_{j=0}^{i-1}(\|a\|\w+\|E\|\f)^j\gamma_{i-j}\\
  & \|F\|\f\le
  \left(\frac12\|a'\|\w^2+\|E\|\f\right)\exp(\|a\|\w+\|E\|\f)
\end{split}
\]
where the constants $\gamma_i$, $i\ge 1$, are defined in \eqref{eq:gammai}.
\end{theorem}

The above theorem states that, also in the case of $A=T(a)+E$, the
matrix $F=\exp(A)-T(\exp(a))$ is such that
$\sum_{i,j\in\ZZ^+}|f_{i,j}|<+\infty$. Moreover, if $E=0$, the bounds
given in the above theorem reduce to the bounds given in Theorem
\ref{mainthm}.

With some formal manipulations, it is possible to provide the
following bound to $\|D_i\|\f$ expressed in closed form
\[
\begin{split}
  &\|D_i\|\f\le \frac{1}{\|E\|\f}
  \left(\varphi\frac{(\|a\|\w+\|E\|\f)^i-\|a\|^i\w}{\|E\|\f}-
    \psi i\|a\|^{i-1}\w\right)\\
  &\varphi=\|a'\|\w^2+\|E\|\f^2,\quad \psi=\|a'\|\w^2
\end{split}
\]
which, taking the limit for $\|E\|\f\to 0$, coincides with the bound
of Theorem \ref{mainthm}.

\section{Algorithms}\label{sec:algo}
In this section we provide an algorithm for computing the exponential
function of a matrix $A$ such that $A=T(a)+E$ with $a(z),a'(z)\in\mathcal W$ and $E$ such that
$\|E\|\f<+\infty$. 

From the condition $a(z),a'(z)\in\mathcal W$ and $E\in\mathcal F$ it follows that
the coefficients $a_i$ of $a(z)$ decay to zero as $i\to\pm\infty$, and the entries $e_{i,j}$ of the matrix $E$ decay to zero as $i,j\to\infty$. Thus we may
represent $a(z)$ in an approximate way just by considering a finite
number of coefficients, i.e, $a(z)=\sum_{i=-n_-}^{n_+}a_iz^i+r(z)$,
where we assume that the remainder $r(z)$ is such that
$\|r\|\w\le\epsilon$, for a given error bound $\epsilon$.  Similarly, also the matrix $E$ can be represented in an approximate way by storing only a finite number of nonzero entries. 
Observe that for the decay of the coefficients $a_i$, 
also the matrix $H(a_-)$ can be represented, up to an error $\epsilon$, by
means of a semi-infinite matrix which is zero everywhere except in its
$n_-\times n_-$ leading principal submatrix, which coincides with the
Hankel matrix associated with $a_{-1},\ldots,a_{-n_-}$.

For computational reasons, it is convenient to represent this
semi-infinite matrix as the product $UV^T$ where $U$ and $V$ have
infinitely many rows and $n_-$ columns. Moreover, due to the
truncation of the series, the matrices $U$ and $V$ have null entries
for a sufficiently large row index.

Define $p_i(z)=\frac 1 {i!} a(z)^i$, $i\ge 0$, so that, for a
sufficiently large $i$, $\exp(a)$ can be approximated by $s_i(z)$,
defined my means of the recursion
\begin{equation}\label{eq:ak}
\begin{split}
&p_i(z)=\frac1i a(z)p_{i-1}(z),\\
&s_i(z)=s_{i-1}(z)+p_{i}(z),~~i\ge 1,
\end{split}\end{equation}
with $s_0(z)=1$.

We consider separately, in the following two sections, the Toeplitz
case, i.e., $A=T(a)$, and the general case where $A=T(a)+E$ with $E\ne
0$.

\subsection{The Toeplitz case}
Consider the case where $A$ is Toeplitz, i.e., $A=T(a)$ and $E=0$.
According to the results of the previous section, the matrix $\exp(A)$
is approximated by $T(s_k)+F_k$, for a suitable $k\ge 1$.  In order to
compute $s_k(z)$ we rely on formula \eqref{eq:ak}, while for computing
$F_k=\sum_{i=0}^k\frac1{i!}E_i$ we rely on recursion \eqref{eq:Ek}.

We define $\widehat E_i=\frac1{i!}E_i$, 
so that we may rewrite equation \eqref{eq:Ek} in the following form
\[
\widehat E_i=\frac 1i T(a) \widehat E_{i-1}- \frac 1 i
H(a_-)H((p_{i-1})_+),~~i\ge1,
\]
so that
$
F_k=\sum_{i=0}^k\widehat E_i
$.

In order to reduce the complexity of the computation, we will
represent also the matrices $\widehat E_i$,
$F_i$ and the matrix $H(a_-)$ in the form
\begin{equation}\label{eq:EF}
\widehat E_i=U_iV_i^T,\quad F_i=W_iY^T_i,\quad 
H(a_-)=UV^T,
\end{equation}
where $U_i,V_i,W_i,Y_i$, $U$ and $V$ are matrices with infinitely many
rows and with a finite number of columns. Moreover, due to the finite
representation, these matrices have null entries if the row index is
sufficiently large. Therefore they can be represented, up to an
arbitrarily small error, with a finite number of parameters.

Using the decompositions \eqref{eq:EF}  we may write
\[
U_iV_i^T=\frac 1i T(a)U_{i-1} V_{i-1}^T-\frac 1i UV^TH((p_{i-1})_+)
\]
whence
\begin{equation}\label{eq:update}
U_i=\begin{bmatrix}
 T(a)U_{i-1} &\brb&  U
\end{bmatrix},\quad 
V_i=\begin{bmatrix}
\frac1i V_{i-1} &\brb& -\frac 1i H(( p_{i-1})_+)V
\end{bmatrix}.
\end{equation}
Moreover, from the relation $F_k=F_{k-1}+U_kV_k^T$ we obtain
\begin{equation}\label{eq:update1}
W_k=[W_{k-1}~~\brb~~ U_k],\quad Y_k=[Y_{k-1}~~\brb~~ V_k].
\end{equation}

By using these decompositions, the implementation of equation
\eqref{eq:Ek}, together with the computation of
$F_k=\sum_{i=1}^k\widehat E_i$ and of the function
$s_k(z)=\sum_{i=0}^k p_i(z)$, will proceed as described in the
following algorithm, where matrices are formed by a finite number of
rows containing the nonzero part of the corresponding infinite
matrices.\medskip

\begin{algorithm}\label{algo:exp}\rm 
\\ {\sc Input:} Integer $k\ge 1$, the coefficient vectors of the
functions $a(z)$, $p_{k-1}(z)$, $s_{k-1}(z)$ and the matrices $U$,
$V$, $U_{k-1}$, $V_{k-1}$, $W_{k-1}$, $Y_{k-1}$, such that
\eqref{eq:EF} holds for $i=k-1$.  \\ {\sc Output:} The coefficient
vectors of the functions $p_{k}(z)$, $s_{k}(z)$ and the matrices $U_
{k}$, $V_{k}$, $W_k$, $Y_{k}$, such that \eqref{eq:EF} holds for
$i=k$.  \\ {\sc Computation:}
\begin{enumerate}
\item compute $P_1=H((p_{k-1})_+)V/k$, set $Q_1=[\frac 1k V_{k-1}
  ~\brb~ -P_1]$
\item compute $P_2=T(a)U_{k-1}$, set $Q_2=[P_2 ~\brb~ U]$
\item\label{compress1} compress the pair $Q_1,Q_2$ and get a new pair
  $V_k,U_k$
\item set $S_1=[W_{k-1} ~\brb~ U_k]$ and $S_2=[Y_{k-1} ~\brb~ V_k]$
\item\label{compress2} compress the pair $S_1,S_2$ and get the new
  pair $W_k,Y_k$
\item compute $p_k(z)=\frac1k a(z)p_{k-1}(z)$ and set
  $s_k(z)=s_{k-1}(z)+p_k(z)$
\item\label{truncate} truncate $s_k(z)$ and $p_k(z)$
\end{enumerate}
\end{algorithm}\medskip

In the above description we have used a compression operation in
stages \ref{compress1} and \ref{compress2} acting on a pair of
matrices, together with the operation of truncating a Laurent
polynomial at stage \ref{truncate}.  We will describe these operations
in the next Subsection \ref{sub:compress}.
 
Observe also that even if the involved matrices have infinitely many
rows, only a finite number of them is nonzero. A detailed
implementation of the above algorithm should keep track of the number
of the nonzero rows of each matrix. We leave this detail to the
reader.

\subsection{The general case}

Consider the case where $A=T(a)+E$, with $E\neq 0$.  According to the
results of Section \ref{sec:gener}, the matrix $\exp(A)$ is
approximated by $T(s_k)+F_k$, for a suitable $k\ge 1$, where
$F_k=\sum_{i=0}^k\frac1{i!}D_i$ and the matrices $D_i$ are defined in
\eqref{eq:Dk0}.
 
As in the previous section, define $\widehat D_i=\frac{1}{i!}D_i$, so
that, in view of \eqref{eq:Dk}, we have
\begin{equation}\label{eq:Dkhat}
  \widehat D_i= 
  \frac{1}{i} A \widehat D_{i-1}-\frac {1}{i} H(a_-)H((p_{i-1})_+)+
  \frac{1}{i}ET(p_{i-1}).
\end{equation}

Let us represent the matrices $E$, $H(a_-)$ and $D_i$ in the form
$E=WY^T$, $H(a_-)=UV^T$ and $\widehat D_i=U_iV_i^T$, where $W$, $Y$,
$U$, $V$, $U_i$ and $V_i$ have a finite number of columns.  We may
rewrite equation \eqref{eq:Dkhat} in the form
\[
U_iV_i^T=\frac{1}{i}(T(a)+WY^T)U_{i-1}V_{i-1}^T
-\frac{1}{i}UV^TH((p_{i-1})_+) +
\frac{1}{i}WY^T T(p_{i-1}).
\]
Whence we deduce that
\[\begin{split}
&
U_i=[(T(a)U_{i-1}+W(Y^TU_{i-1})~~\brb~~ U~~\brb~~ W],\\
& V_i=\left[\frac{1}{i}V_{i-1}~~\brb~~ 
-\frac{1}{i}H((p_{i-1})_+)V
~~\brb~~ 
\frac{1}{i} T(p_{i-1})^T Y
\right].
\end{split}\]

Moreover, by representing $F_k$ as $F_k=W_k Y_k^T$,
from the relation $F_k=F_{k-1}+U_kV_k^T$ we obtain
\[
W_k=[W_{k-1}~~\brb~~ U_k],\quad Y_k=[Y_{k-1}~~\brb~~ V_k].
\]

It is immediate to write an algorithm that implements the above
equations. We leave this task to the reader.

\subsection{Compression and truncation}\label{sub:compress}
Given the matrix $E$ in the form $E=FG^T$ where $F$ and $G$ are
matrices of size $m\times k$ and $n\times k$, respectively, we aim to
reduce the size $k$ and to approximate $E$ in the form $E=\widetilde
F\widetilde G^T$ where $\widetilde F$ and $\widetilde G$ are matrices
of size $m\times \tilde k$ and $n\times \tilde k$, respectively, with
$\tilde k\le k$.

We use the following procedure which, for simplicity, we describe in
the case of real matrices. Compute the pivoted (rank-revealing) QR
factorizations $F= Q_fR_fP_f$, $G= Q_gR_gP_g$, where $P_f$ and $P_g$
are permutation matrices, $Q_f$ and $Q_g$ are orthogonal and $R_f$,
$R_g$ are upper triangular with columns having non increasing
Euclidean norm. Remove the last negligible rows from the matrices
$R_f$ and $R_g$ and remove the corresponding columns of $Q_f$ and
$Q_g$.  In this way we obtain matrices $\hat R_f$, $\hat R_g$, $\hat
Q_f$, $\hat Q_g$ such that, up to within a small error, satisfy the
equation $F= \hat Q_f\hat R_fP_f$, $G= \hat Q_g\hat R_gP_g$. Then, in
the factorization $FG^T=\hat Q_f(\hat R_fP_fP_g^T\hat R_g^T)\hat
Q_g^T$, compute the SVD of the matrix in the middle $ \hat
R_fP_fP_g^T\hat R_g^T=U\Sigma V^T$ where the singular values
$\sigma_i$ satisfying the condition $\sigma_i<\epsilon\sigma_1$ are
removed together with the corresponding columns of $U$ and $V$, where
$\epsilon$ is a given tolerance, say the machine precision. In output,
the matrices $\widetilde F=\hat Q_f U\Sigma^{1/2}$, $\widetilde G=\hat
Q_gV\Sigma^{1/2}$ are delivered.

This procedure is described with more detail in the next\medskip

\begin{algorithm}\label{algo:compress}\rm
\\ {\sc Input:} Matrices $F$ and $G$ of size $m\times k$ and $n\times
k$, respectively, a real $\epsilon>0$; \\ {\sc Output:} matrices
$\widetilde F$ and $\widetilde G$ of size $ m\times \widetilde k$ and
$ n\times \widetilde k$, respectively, such that $\widetilde k\le k$
and $\|FG^T-\widetilde F\widetilde G^T\|_2\le \gamma
\epsilon\|F\|_2\|G\|_2$ for a suitable $\gamma>0$.  \\ {\sc
  Computation:}
\begin{enumerate}
\item Compute the pivoted (rank-revealing) QR factorizations
$F= Q_fR_fP_f$, $G= Q_gR_gP_g$;
\item select $h_f$ and $h_g$ the smallest integers such that
  $|(R_f)_{i,i}|<\epsilon |(R_f)_{1,1}|$ for $i>h_f$ and
  $|(R_g)_{i,i}|<\epsilon |(R_g)_{1,1}|$ for $i>h_g$;
\item denote $\hat R_f$, $\hat R_g$, the submatrices of $R_f$ and
  $R_g$ formed by the first $h_f$ rows and by the first $h_g$ rows,
  respectively; denote $\hat Q_f$, $\hat Q_g$ the submatrices of $Q_f$
  and $Q_g$ formed by the first $h_f$ and $h_g$ columns respectively;
\item compute the SVD of $ \hat R_fP_fP_g^T\hat R_g^T$, i.e., $\hat
  R_fP_fP_g^T\hat R_g^T=U\Sigma V^T$;
\item select the integer $\ell$ such that $\sigma_i<\epsilon\sigma_1$
  for $i>\ell$ where $\sigma_i$ are the singular values, and set
  $\widehat U$, $\widehat V$ the submatrices formed by the first
  $\ell$ columns of $U$ and $V$, respectively; set $\widehat
  \Sigma=\hbox{diag}(\sigma_1,\ldots,\sigma_\ell)$ so that $\|U\Sigma
  V^T-\widehat U\widehat\Sigma\widehat V^T\|_2\le \sigma_{\ell+1}$;
\item output $\widetilde F=\widehat Q_f\widehat U\widehat
  \Sigma^\frac12$, $\widetilde G=\widehat Q_g\widehat V\widehat
  \Sigma^\frac12$.
\end{enumerate}
\end{algorithm}
\medskip

One can show that the compression obtained this way is such that
$\|FG^T-\widetilde F\widetilde G^T\|_2\le\gamma\epsilon\|F\|_2\|G\|_2$
where $\gamma $ is a positive constant depending on the numerical
rank.

\subsection{Acceleration}
The scaling techniques for accelerating convergence of the exponential
series described in \cite{high:book} can be easily implemented in this
framework.  In particular, in the case where $A=T(a)$, we determine
the least integer $q$ such that $\|a\|\w/2^q<1$, then we set $\widehat
a(z)=a(z)/2^q$ so that
\[
\exp(T(a))=\exp(T(\widehat a))^{2^q}
\] 
and we may compute $\exp(T(a))$ by first computing $\exp(T(\widehat
a))$, which requires a shorter power series expansion, and then
computing $\exp(T(\widehat a))^{2^q}$ by means of $q$ steps of
repeated squarings applied to $\exp(T(\widehat a))$.

The square of a matrix of the kind $T(a)+E$ is computed by means of
the equation
\[
(T(a)+E)^2=T(a^2)-H(a_-)H(a_+)+T(a)E+ET(a)+E^2=:T(a^2)+\widehat E
\]
where
$\widehat E=-H(a_-)H(a_+)+T(a)E+ET(a)+E^2$.

Assuming that $E$ is factored in the form $E=WY^T$, for $W$ and $Y$
being slim matrices, and that $H(a_-)$ is factored as $UV^T$, being
$U$ and $V$ slim, then $\widehat E$ is factored as $\widehat
E=\widehat U\widehat V^T$ where
\begin{equation}\label{eq:hUV}
  \widehat U=[-U ~\brb~ T(a)W ~\brb~ W],\quad 
  \widehat V=[H(a_+)V ~\brb~Y ~\brb~ T(a)^TY+Y(W^TY)].
\end{equation}
A compression step, performed according to Algorithm
\ref{algo:compress} can be applied to reduce the rank of $\widehat U$
and $\widehat V$. In this case, since some of the involved matrices
are Hankel, one can try to get advantage from this property. Observe
that, if the numerical rank of the matrix $H(a_+)$ is smaller than
that of $H(a_-)$, it's more convenient to express $H(a_+)$ in the form
$H(a_+)=UV^T$ where $U$ and $V$ are suitable slim matrices.

The advantage that we obtain by means of the scaling and squaring
technique is substantial in many cases.

\subsection{Computational analysis}
We may perform a complexity analysis of the algorithms designed in the
previous sections. Here we give an outline of this analysis and we
leave to the reader the completion of the details.

We consider only the case where $A=T(a)$ and divide the problem into
the different sub-problems of evaluating the recurrence \eqref{eq:Ek}
by means of equations \eqref{eq:update} and \eqref{eq:update1},
performing the compression according to Algorithm \ref{algo:compress},
and computing the repeated squaring of a QT matrix.

Concerning \eqref{eq:update}, we have to compute the product
$T(a)U_{i-1}$, where $T(a)$ is an infinite Toeplitz matrix having
bandwidth $n_-+n_+$, and $U_{i-1}$ has infinitely many rows and a
finite number, say $r_{i-1}$, of columns.  Denoting $m_{i-1}$ the
number of numerically nonzero rows of $U_{i-1}$, the problem is
reduced to multiplying an $(m_{i-1}+n_-)\times m_{i-1}$ Toeplitz
matrix and an $m_{i-1}\times r_{i-1}$ matrix. By using fast algorithms
for Toeplitz-vector matrix multiplication we have a cost of
$O(r_{i-1}(m_{i-1}+n_-)\log (m_{i-1}+n_-))$ arithmetic operations
(ops). Similarly, the computation of the product $H((p_{i-1})_+)V$ is
reduced to multiplying an $n_{i-1}\times q_i$ Hankel matrix times a
matrix of size $q_i\times s$, where $n_{i-1}$ is the degree of the
polynomial $(p_{i-1}(z))_+$, $s$ is the number of columns of $V$,
$q_i=\min(n_{i-1},n)$, with $n$ the number of numerically nonzero rows
of $V$.  Thus, even this computation has a cost of $O(sn_{i-1}\log
n_{i-1}))$.  In fact, the product of a Hankel matrix and a vector, up
to permutation, is the same as the product of a Toeplitz matrix and a
vector where FFT-based algorithms can be used.

The cost of compression in the steps 3 and 5 of Algorithm
\ref{algo:exp} performed with Algorithm \ref{algo:compress}, which
relies on QR and SVD, is proportional to the square of the rank and to
the maximum dimension. Finally, the cost of repeated squarings, once
the rank-revealing factorization $H(a_-)=UV^T$ has been computed, is
dominated by the compression of the matrices $\widehat U$ and
$\widehat V$ defined in \eqref{eq:hUV} which is again proportional to
the square of the rank and to the maximum matrix size. On the other
hand, the computation of $U$, $V$, such that $H(a_-)=UV^T$ involves a
rank revealing factorization of the Hankel matrix $H(a_-)$. This
factorization is computed by means of a Lanczos-based
tridiagonalization with re-orthogonalization which exploits the Hankel
structure and the low cost of matrix-vector product. Details on this
computation are given in \cite{bmm:sbornik} where it is shown that the
cost is proportional to the square of the rank and to $n\log n$ where
$n$ is the size of the Hankel matrix.

Roughly speaking, the overall cost is proportional to the number of
terms in the series, to the square of the maximum rank of the
corrections and to the maximum between the band width of $T(a)$ and
the size of the correction.  In particular, the lower the rank the
faster the algorithm.

We may perform a more detailed analysis of the errors generated by the
truncation part of the algorithm. This analysis concerns the
estimation of the constant $\gamma$ given in Algorithm
\ref{algo:compress} and it consists of giving technical bounds to the
norm of the submatrices involved in the compression by exploiting the
properties of rank revealing QR and of SVD. We avoid to provide this
analysis and leave it to the reader.

\section{Numerical experiments}
We have implemented in Matlab the algorithm for computing
$\exp(T(a))$, relying on the scaling and squaring acceleration, and
performed some numerical experiments with several problems. All the
experiments have been run under the Linux system on a I3 processor
with Matlab R2016a.

A first bunch of tests concerns the computation of the exponential of
symmetric tridiagonal Toeplitz matrices $T=\hbox{trid}(1,\alpha,1)$
for different values of $\alpha$ in the range $[-4,4]$. From the
experiments it turns out that the numerical bandwidth of the Toeplitz
part, the size and the rank of the correction, as well as the CPU
time, are independent of the values of $\alpha$. The time is roughly
$0.01$ seconds, the numerical bandwidth is 35 while the correction has
size $16\times 16$ and rank 7. The relative norm errors of the
approximation are between $3\cdot 10^{-15}$ and $1.0\cdot 10^{-14}$.

Another test that we have performed concerns Toeplitz matrices
associated with the Laurent polynomial
$a(z)=\sum_{i=1}^{n_+}z^i+\sum_{i=0}^{n_-}z^{-i}$ where $T(a)$ is a
banded matrix with $n_-$ and $n_+$ diagonals below and above the main
diagonal, respectively. Here the matrix $\exp(T(a))$ has a numerical
bandwidth which grows large the larger are $n_-$ and $n_+$. In this
case, the approximation of $\exp(T(a))$ by means of the finite
exponential $\exp(T_n(a))$, where $T_n(a)$ denotes the $n\times n$
leading principal submatrix of $T(a)$, requires that $n$ is as large
as at least the numerical bandwidth of $\exp(T(a))$.

We tested the cases obtained with $n_+=5$, for different values of
$n_-$ where, concerning $T_n(a)$, we have chosen $n=2m$ where $m$ is
the numerical bandwidth of $T(\exp(a))$.  We have doubled the value of
$m$ in order to remove the boundary effect in the computation of the
finite exponential. Table \ref{tab1} reports, the value of $n_-$, the
time $t_{QT}$ needed by our Matlab function to compute $\exp(T(a))$
and the time $t_{\tt expm}$ needed by the Matlab function {\tt expm}
for computing the matrix exponential of $T_n(a)$, and the relative
error in norm
$\|\exp(T(a))_m-\exp(T_n(a))_m\|_\infty/\|\exp(T(A))_m\|_\infty$,
together with the numerical bandwidth of $T(\exp(a(z)))$, the size and
the rank of the correction $F$ such that $\exp(T(a))=T(\exp(a))+F$.

\begin{table}\begin{center}
\begin{tabular}{r|r|r|r|r|r|r|r}
$n_-$&$t_{QT}$&$t_{\tt expm}$&err&band&rows&cols&rank\\ \hline
10 &  0.05 & 0.32 & 2.3e-14 & 331 & 372 & 245 & 26 \\
20 &  0.10 & 3.52 & 6.6e-14 & 831 & 752 & 271 & 23  \\
30 &  0.10 & 20.70 & 2.1e-13 & 1519 & 1708 & 291 & 18  \\
40 &  0.12 & 102.99 & 2.5e-13 & 2377 & 2948 & 230 & 11  \\
50 &  0.15 & - & - & 3393 & 3343 & 214 & 10  \\
60 &  0.22 & - & - & 4563 & 4490 & 267 & 10  \\
70 &  0.22 & - & - & 5881 & 5827 & 50 & 9  \\
80 &  0.29 & - & - & 7343 & 7283 & 49 & 9  \\
90 &  0.30 & - & - & 8947 & 8867 & 47 & 9  \\
100 &  0.38 & - & - & 10689 & 13383 & 45 & 8  \\
\end{tabular}\caption{Exponential of a Toeplitz matrix having 5 
  diagonals of 1s in the upper triangular part and $n_-$ diagonals 
  of 1s in the lower triangular part. Comparisons of the CPU time 
  needed by the algorithm based on QT matrices and by the Matlab 
  {\tt expm} function. Relative errors in norm of the values computed 
  in the two different ways, band-width of the Toeplitz part, size and 
  rank of the correction $F$ in the matrix exponential 
  $\exp(T(a))$.}\label{tab1}
\end{center}
\end{table}

We can see from the table the different growth of the CPU time for the
two algorithms and the small value of the rank of the correction,
despite the large size of the band. The error is reported only for
$n_-\le 40$ since for larger values of $n_-$ a memory overflow is
encountered by the function {\tt expm}.

It is interesting to observe that, roughly, the CPU time grows
proportionally to the maximum size of the correction and of the
bandwidth, times the square of the rank.

In order to figure out how the size and the rank of the correction
grow in the partial sum $S_k=\sum_{i=0}^k\frac 1{i!}A^i$, as $k$
increases, we considered the case where $a(z)=\sum_{j=-20}^{20}a_jz^j$
and the coefficients $a_j$ are randomly distributed in $[0,1]$. In
Figure \ref{fig:corr} we report the graphs of the number of nonzero
rows, the number of nonzero columns and the numerical rank of the
correction, as function of $k$, in the power $T(a)^k$, in the partial
sum $Y_k$ and in the matrices generated at any step of the repeated
squaring part of the algorithm.

It is interesting to point out that, while the number of nonzero rows
and columns in the correction matrices grows with $k$ almost linearly,
the value of the numerical rank does not grow much with $k$. This
property makes it cheap to update the correction $E_k$ by means of
\eqref{eq:update} and \eqref{eq:update1}, and to perform compression
according to Algorithm~\ref{algo:compress}.

\pgfdeclareimage[width=10cm]{corr_power}{corr_rand_20_20_pow}
\pgfdeclareimage[width=10cm]{corr_sum}{corr_rand_20_20_sum}
\pgfdeclareimage[width=10cm]{corr_squaring}{corr_rand_20_20_square}

\begin{figure}
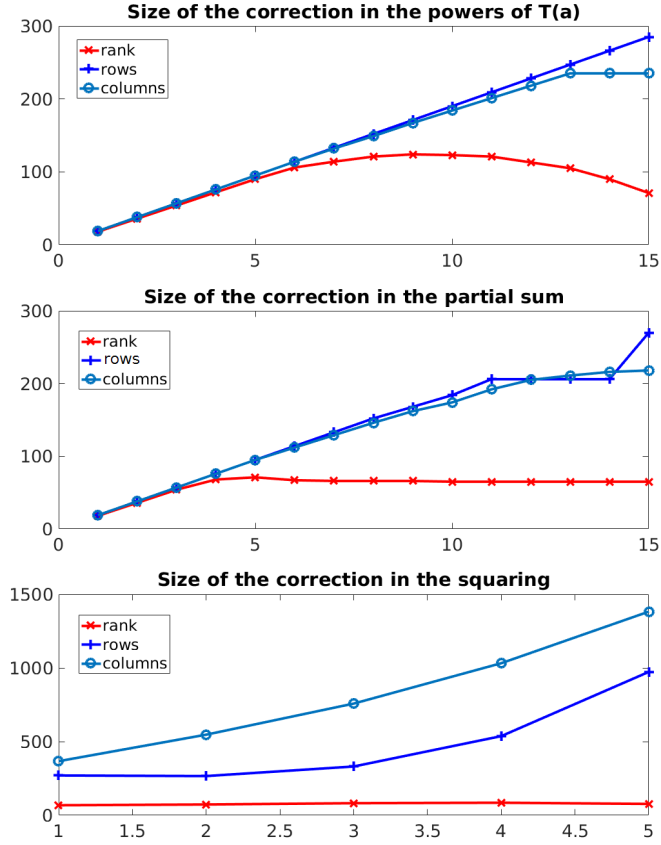

\begin{center}
\begin{tabular}{c}
\pgfuseimage{corr_power}\\
\pgfuseimage{corr_sum}\\
\pgfuseimage{corr_squaring}
\end{tabular}\caption{Growth, as function of $k$, of the number of rows/columns and
  of the numerical rank in the correction matrix of $T(a)^k$, $S_k$, and of the
  matrices generated by the squaring phase of the algorithm. The
  function $a(z)=\sum_{i=-20}^{20}a_iz^i$ has random coefficients $a_i$
  uniformly distributed between $0$ and $1$.}\label{fig:corr}
\end{center}
\end{figure}

Another test that we have performed concerns the option pricing
problem using the Merton model of \cite{kressner}, \cite{merton}. In
this case, the exponential of a finite $n\times n$ Toeplitz matrix
$T_n$ has to be computed. This matrix can be associated with a symbol
$a^{(n)}(z)$ which depends on the value of $n$. In our experiment, we
considered, as test problem, the semi-infinite Toeplitz matrix
$T(a^{(n)})$ having as leading principal submatrix the $n\times n$
Toeplitz matrix $T_n$ of the Merton model.  We have compared the two
$\frac n2\times\frac n2$ leading principal submatrices of
$\exp(T(a^{(n)}))$ and $\exp(T_n)$ just to verify if the two different
matrices have some similarity.  We computed $\exp(T_n)$ with the
function {\tt expm}, and $\exp(T(a^{(n)}))$ with our algorithm. In
Table \ref{tab:merton} we report the CPU time of this computation and
the relative errors in norm concerning the principal submatrices of
size $\frac n2$ of the two matrices, and the rank of the correction
$F$.  We can see that the CPU time grows slightly more than linearly
with the size $n$ and that the rank of the correction does not grow
with $n$. Also the error, computed only if the size is lower than
4096, takes small values.

\begin{table}
\begin{center}
\begin{tabular}{c|c|c|c|c}
$n$& $t_{QT}$& $t_{\tt expm}$&err&rank\\ \hline
512 &  0.34 & 0.16 & 2.7e-12 & 18 \\
1024 &  0.72 & 1.33 & 2.8e-11 & 18 \\
2048 &  2.16 & 6.43 & 3.6e-10 & 18 \\
4096 &  5.31 & - & - & 19 \\
8192 &  17.18 & - & - & 19 \\
\end{tabular}
\end{center}\caption{CPU time, error, and rank in the Merton model.}\label{tab:merton}
\end{table}

\section{Final remarks}
A classical approach to improve the computation of the matrix
exponential is to rely on Pad\`e approximation \cite{high:book}. This
technique consists in approximating the exponential function by means
of a ratio of polynomials of low degree. We can apply this technique
in our framework provided that an efficient method for computing
inverses of QT matrices is designed. This is possible by relying on
the Wiener-Hopf factorization of the symbol associated with the
Toeplitz part of the matrix. Details on inverting QT matrices are
given in \cite{bmm}.

The approach that we have presented in this paper can be applied to
deal with general matrix functions and with Toeplitz matrices of
finite size $n$ where $n$ is sufficiently large with respect to the
decay of the coefficients of $a(z)$. For simplicity consider the case
of a Laurent polynomial $a(z)=\sum_{i=-n_-}^{n_+}a_iz^i$ for
$n_-,n_+>0$, and the associated $n\times n$ matrix $T_n(a)=(t_{i,j})$,
$t_{i,j}=a_{j-i}$ if $-n_-\le j-i\le n_+$, and $t_{i,j}=0$
otherwise. Observe that Theorem \ref{thm1}, reformulated for finite
matrices and applied to $T_n(a)^2$, leads to the following equation
\[
T_n(a)^2=T_n(a^2)-H_n(a_-)H_n(a_+)-JH_n(a_+)H_n(a_-)J
\]
where $J$ is the flip matrix having ones on the anti-diagonal and
zeros elsewhere, and $H_n(b)$ denotes the $n\times n$ leading
principal submatrix of $H(b)$.

If $n$ is sufficiently large, i.e., $n>n_-+n_+$, then the matrices
$H_n(a_-)H_n(a_+)$ and $JH_n(a_+)H_n(a_-)J$ have disjoint supports
contained in the upper leftmost corner and in the lower rightmost
corner, respectively. Thus, $T_n(a)^2$ can be represented as the sum
of the Toeplitz matrix associated with the Laurent polynomial $a^2(z)$
and a correction $E$ which involves a (small) finite number of nonzero
entries located in two opposite corners of the support. The same
property holds for the powers $T_n(a)^k$ for the values of $k$ for
which the size of the numerical support of the correction $E$ does not
grow much.  Here, for numerical support of a matrix $A=(a_{i,j})$ we
mean the set of indices $(i,j)$ for which $|a_{i,j}|<\epsilon \|A\|$
for some norm.

Thus, if the exponential decay of the coefficients of $a(z)$ and of
$\exp(a(z))$ is sufficiently large with respect to the size $n$, then
representing the powers $T_n(a)^k$ as well as $\exp(T_n(a))$ in the
above form may be computationally effective.

Algorithms for dealing with the finite case can be easily obtained
from the algorithms presented in Section \ref{sec:algo} by repeating
the computation for the correction in the lower rightmost corner
involving the matrices $JH_n(a_+)H_n(a_-)J$.

More details in this regard can be found in \cite{bmm:sbornik} where
it is performed the computational analysis of general functions of
finite and infinite Toeplitz matrices expressed either in terms of
power series or of Cauchy integrals.

\subsection{Conclusions}
We have provided a framework to compute matrix functions of a quasi
Toeplitz matrix based on the fact that matrices of the kind $A=T(a)+E$
form a matrix algebra if $a(z)$ and $a'(z)$ belong to the Wiener class 
 and
$\|E\|\f<+\infty$.  A specific analysis has been performed for the
exponential function.  Numerical experiments confirm the effectiveness
of our approach.  This framework can be applied to the case of finite
matrices and extended to the case of functions expressed by means of a
Cauchy integral. 

This analysis has set some open issues like analyzing the growth, as a
function of $k$, of the numerical rank of the correction $E_k$ such
that $T(a)^k=T(a^k)+E_k$. From the numerical experiments performed so
far with some function $a(z)$ this rank seems to be a bounded function
of $k$. Indeed, the growth of the numerical rank of $E_k$ is related
to the decay of the coefficients of the function $a(z)$ and it is
worth being investigated.

\section*{Acknowledgments}
Dario Bini wishes to thank Bernd Beckermann and Daniel Kressner for very helpful discussions on topics related to the subject of this paper.

\bibliographystyle{abbrv}

\end{document}